\documentclass[a4paper,12pt]{amsart}

\usepackage[centertags]{amsmath}
\usepackage{amsfonts,amssymb,amstext,amsthm,newlfont,latexsym,stmaryrd}
\usepackage[left=3.0cm,right=3.0cm,top=3.7cm,bottom=3.7cm]{geometry}
\usepackage{tikz,hyperref,pgflibraryshapes}
\usetikzlibrary{arrows,automata,positioning,fit}

\theoremstyle{plain}
\newtheorem{theorem}{Theorem}[section]
\newtheorem{corollary}[theorem]{Corollary}
\newtheorem{lemma}[theorem]{Lemma}
\newtheorem{proposition}[theorem]{Proposition}
\theoremstyle{definition}
\newtheorem{definition}[theorem]{Definition}
\newtheorem{example}[theorem]{Example}
\theoremstyle{remark}
\newtheorem{remark}[theorem]{Remark}
\newtheorem*{acknowledgement*}{Acknowledgment}

\begin{document}

\title[Class-closing factor codes]{Class-closing factor codes and constant-class-to-one factor codes from shifts of finite type}

\author[M. Allahbakhshi]{Mahsa Allahbakhshi}
\address{Universidad de Santiago de Chile\\
    Alameda 3363\\
    Estaci{\'o}n Central\\
    Santiago de Chile\\
    Chile}
\email{mallahbakhshi@dim.uchile.cl}

\author[S. Hong]{Soonjo Hong}
\address{Centro de Modelamiento Matem\'atico \\
    Universidad de Chile \\
    Av. Blanco Encalada 2120, Piso 7 \\
    Santiago de Chile \\
    Chile}
\email{hsoonjo@dim.uchile.cl}

\author[U. Jung]{Uijin Jung}
\address{Department of Mathematics \\
	Ajou University \\
    Suwon 443-749 \\
	South Korea}
\email{uijin@ajou.ac.kr}

\date{}
\begin{abstract}
  We define class-closing factor codes from shifts of finite type and show that they are continuing if their images are of finite type. We establish several relations between class-closing factor codes, continuing factor codes and constant-class-to-one factor codes. In particular it is shown that a factor code between irreducible shifts of finite type is constant-class-to-one if and only if it is bi-class-closing, generalising a result of Nasu.
\end{abstract}
\maketitle

\section{Introduction}
  
  It is well known that given a finite-to-one factor code $\pi$ from a shift of finite type $X$ onto an irreducible sofic shift $Y$, almost all points in $Y$ have the same number of preimages. This number is called the \emph{degree} of $\pi$, is invariant under topological conjugacy and is essential in the study of finite-to-one codes \cite{Hed69, LinM95}. When $\pi$ is allowed to be infinite-to-one, the \emph{class degree} introduced in \cite{AllQ13} may be thought as a natural generalisation of the degree. For each $y$ in $Y$, define an equivalence relation on the fibre $\pi^{-1}(y)$ as follows: two points $x$ and $x'$ in $\pi^{-1}(y)$ are equivalent if there is another point $z$ in the same fibre which equals $x$ up to an arbitrarily large given coordinate and is right asymptotic to $x'$ and vice versa. Then the number of equivalence classes (called \emph{transition classes}) are finite for all $y$ in $Y$ and each right transitive point in $Y$ has the same number of classes. This number is called the \emph{class degree} of $\pi$ and equals to the degree when $\pi$ is finite-to-one \cite{AllQ13}.
  
  The class degree was initially devised in the study of the ergodic measures of relative maximal entropy \cite{AllQ13}. It is a conjugacy-invariant upper bound on the number of such measures over a fully supported ergodic measure. Although there are many freedom in the construction of infinite-to-one codes \cite{Boy83, BoyT84, Jun11}, it turned out that as for a finite-to-one code, fibres over almost all points for an infinite-to-one code are well-behaved \cite{AllHJ13}. Class degree inherits many important properties of the degree. A point $y$ in $Y$ is right transitive if and only if each transition class over $y$ contains a right transitive point. Also, any distinct transition classes over a right transitive point are uniformly separated by some constant. Other structural similarities were investigated and can be found in \cite{AllHJ13}.

  In \cite{AllHJ13,AllQ13}, the transition classes investigated are indeed {\em right} transition classes in this paper. The definition of a transition is asymmetric in time, so we may also consider the {\em left} transition classes. To concentrate only on typical points of $Y$, right transition classes are sufficient, as the set of right transition classes coincide with that of left ones for a doubly transitive point \cite[\S6]{AllHJ13}. However, to consider atypical points, we are naturally led to consider {\em left} transitions. For an example, say that a code is {\em constant-class-to-one} if each point of $Y$ has exactly the same number of (right) transition classes. It is defined in terms of only right transition classes. However, a factor code between irreducible shifts of finite type is constant-class-to-one if and only if the left transition classes and the right ones coincide for each $y$ in $Y$, or equivalently, each point of $Y$ has the same number of left transition classes (see Corollary \ref{cor:separate_right_and_left_classes} and Corollary \ref{cor:bi-class-closing_sft_iff_constant-to-one}).

  In this paper, we continue the investigation of infinite-to-one factor codes through transition classes and class degrees. By introducing the notion of a \emph{tangled set} of words, in \S 3 we give an alternative approach to the class degree \cite{AllQ13}.

  In \S 4, we generalise {\em closing} factor codes, the set of which forms an important class of finite-to-one factor codes, into {\em class-closing} factor codes. A code is called \emph{right class-closing} if any two left asymptotic points in $X$ with the same image are equivalent, rather than the same. If the code is finite-to-one, then a class-closing code is same as a closing code. One reason for which closing factor codes are useful comes from the fact that most crucial properties of any given sofic shift space, such as entropy or mixing property, are lifted to the edge shift determined by the canonical right (resp., left) resolving presentation of the system \cite{Kri84}. Moreover, all the known construction of finite-to-one factor codes between shifts of finite type involve closing codes. Their existence, properties and relations with other invariants such as dimension groups have been studied in depth \cite{Ash90,Ash93,BoyMT87,KitMT91,Nas83}. Several results are presented in \S 4 for class-closing factor codes, which are allowed to be infinite-to-one. They uniformly separate the transition classes over all points in $Y$ and possess the delays as for closing codes (see Theorem \ref{thm:class-closing_conditions}). One known generalisation of a closing code is a {\em continuing code} introduced in \cite{BoyT84}. We also show that a right (resp. left) class-closing factor code between irreducible shifts of finite type is right (resp. left) continuing (see Theorem \ref{thm:class-closing_is_continuing}).

  In \S 5, we investigate the structure of constant-class-to-one codes and present relations to class-closing codes. The structures of bi-closing factor codes are well studied. A finite-to-one factor code between irreducible shifts of finite type is constant-to-one if and only if it is bi-closing if and only if it is open \cite{CovP77, Jun09, Nas83}. generalising this result, we show that a factor code between irreducible shifts of finite type is constant-class-to-one if and only if it is bi-class-closing (see Corollary \ref{cor:bi-class-closing_sft_iff_constant-to-one}). In \cite{Jun11} it was shown that a factor code between shifts of finite type is bi-continuing if and only if it is open. From our result, now we know what lacks for a bi-continuing code to be a constant-class-to-one code. Other properties on bi-class-closing codes and constant-class-to-one codes will be also provided.
  
\section{Backgrounds}

  Basic notations and results on symbolic dynamics are reviewed here. For more expositions, we refer to \cite{LinM95}.

  An {\em alphabet} $\mathcal{A}$ is a finite set whose elements are called {\em symbols}. A {\em word} or {\em block} $w=w|_{[1,n]}$ of length $n\in\mathbb{Z}^+$ over $\mathcal{A}$ is a concatenation of $n$ symbols $w|_j,1\le j\le n$, chosen from $\mathcal{A}$. The {\em empty} word of length 0 is assumed to exist. A {\em shift space} $X$ over $\mathcal{A}$ is a closed $\sigma$-invariant subset of $\mathcal{A}^\mathbb{Z}$, endowed with the product topology, where the {\em shift map} $\sigma:\mathcal{A}^\mathbb{Z}\to\mathcal{A}^\mathbb{Z}$ is given by $\sigma(x)|_i=x|_{i+1}$. If for any $u,v$ in the set $\mathcal{B}(X)$ of all words occurring in the points of $X$ there is a word $w$ with $uwv$ in $\mathcal{B}(X)$, then $X$ is said to be {\em irreducible}. Let $\mathcal{B} _n(X)=\mathcal{B}(X)\cap\mathcal{A}^n$.

  A point $x$ of $X$ is {\em right transitive} if every block in $\mathcal{B}(X)$ occurs in $x|_{[0,\infty)}$.  Two points $x$ and $x'$ of $X$ are said to be {\em right asymptotic} if $x|_{[n,\infty)}=x'|_{[n,\infty)}$ for some $n\in\mathbb Z$. {\em Left transitive points} and {\em left asymptotic points} are defined analogously. Two points $x$ and $x'$ in a shift space are {\em mutually separated} if $x|_i\ne x'|_i$ for each $i\in\mathbb{Z}$.

  We call a continuous $\sigma$-commuting map between shift spaces a {\em code}, a surjective one a {\em factor code} and a bijective one a {\em conjugacy}. Shift spaces $X$ and $Y$ are {\em conjugate} if there is a conjugacy between them. A code $\pi:X\to Y$ is {\em 1-block} if $x|_0$ determines $\pi(x)|_0$ for any point $x$ in $X$. A 1-block code $\pi:X \to Y$ induces a map from $\mathcal{B}(X)$ into $\mathcal{B}(Y)$, also denoted by $\pi$. Every code $\pi:X\to Y$ is {\em recoded} to some 1-block code $\tilde\pi:\tilde X\to Y$, where $\tilde X$ is conjugate to $X$ by a conjugacy commuting with $\tilde\pi$. We call $\pi$ {\em constant-to-one} if $|\pi^{-1}(y)|$ is constant for all $y$ in $Y$ and {\em finite-to-one} if $|\pi^{-1}(y)|$ is finite for all $y$ in $Y$. Otherwise it is called  {\em infinite-to-one}. An image of an irreducible shift space under a factor code is also irreducible.
  
  A factor code $\pi:X\to Y$ is {\em right closing} if it does not allow two distinct left asymptotic points with the same image. A 1-block factor code $\pi:X\to Y$ is {\em right resolving} if $\pi(ab)=\pi(ac)$ implies $b=c$ for all $ab,ac$ in $\mathcal{B}_2(X)$. {\em Left closing} and {\em left resolving} maps are defined analogously. A factor code is {\em bi-closing} (resp., {\em bi-resolving}) if it is both right and left closing (resp., resolving). It is easy to see that a right resolving factor code is right closing and a right closing factor code is conjugate to a right resolving factor code.
  
  A word $w$ in $\mathcal{B}(X)$ is said to be {\em synchronizing} if whenever $uw$ and $wv$ are in $\mathcal{B}(X)$, so is $uwv$.  If for some $k\ge0$ every word of length $k$ in $\mathcal{B}(X)$ is synchronizing, then $X$ is called a ({\em $k$-step}) {\em shift of finite type}. Every shift of finite type is conjugate to the {\em edge shift} $\mathsf{X}_G$ over a directed graph $G$, where $G=(\mathcal{V},\mathcal{E})$, a pair of finitely many {\em vertices} $\mathcal{V}$ and {\em (directed) edges} $\mathcal{E}$, determines the shift space $\mathsf{X}_G$ consisting of all bi-infinite paths over $G$. Every edge shift is 1-step. For $\Gamma\subset\mathcal{B}(\mathsf{X}_G)$, denote by $\mathsf{i}_G(\Gamma)$ and by $\mathsf{t}_G(\Gamma)$ the set of the initial vertices and that of the terminal vertices of the paths in $\Gamma$, respectively. A graph is said to be {\em irreducible} if for each pair $I, J$ in $\mathcal{V}$ there is a path from $I$ to $J$ in $G$. A shift of finite type is irreducible if and only if it is conjugate to $\mathsf{X}_G$ for some irreducible graph $G$.
  
  A {\em sofic shift} $Y$ is a factor of a shift of finite type $X$. We call $X$ an {\em extension} of $Y$. Any sofic shift is the image of a 1-block factor code from an edge shift. This 1-block factor from the edges of a graph is called a {\em labelling} map.

\section{Transition class and class degree}\label{sec:class_degrees}
  We review transition classes and class degrees and find useful facts as well as give new proofs of some old results. For more details on class degree, see \cite{AllHJ13,AllQ13}. From now on, let $\pi:X\to Y$ be a factor code from a shift of finite type $X$ over $\mathcal{A}$ onto a sofic shift $Y$, unless stated otherwise.

  \begin{definition}\label{defn:bridge}
    Let $u$ and $w$ be in $\mathcal{B}_l(X)$ for some $l>0$ and $\pi(u)=\pi(w)$. Then a path $v$ in $\mathcal{B}(X)$ is called a {\em bridge} from $u$ to $w$ if $v|_1=u|_1,v|_l=w|_l$ and $\pi(v)=\pi(u)=\pi(w)$. A pair of bridges from $u$ to $w$ and from $w$ to $u$ is called a {\em 2-way} {\em bridge} between $u$ and $w$.
  \end{definition}

  \begin{definition}\label{defn:class-degree}
    Given $m\in\mathbb Z,y$ in $Y$ and $x,x'$ in $X$ with $\pi(x)=\pi(x')$ a {\em right} $m$-{\em bridge} from $x$ to $x'$ is a point $\vec x$ in $X$ such that for some $n>m$ we have
    \[ x|_{(-\infty,m]}=\vec x|_{(-\infty,m]},\vec x|_{[n,\infty)}=x'|_{[n,\infty)}\text{ and }\pi(x)=\pi(x')=\pi(\vec x). \]
    A pair of right $m$-bridges from $x$ to $x'$ and from $x'$ to $x$ is called a {\em 2-way} $m$-{\em bridge} between $x$ and $x'$. A {\em right transition} from $x$ to $x'$ is a sequence $\{\vec x^{(m)}\}_{m\in\mathbb{Z}}$ of $m$-bridges from $x$ to $x'$. When there is a {\em right transition} from $x$ to $x'$, we write that $x\to^rx'$.

    We say that $x$ and $x'$ are {\em right} {\em equivalent} and write that $x\sim^rx'$ if $x\to^rx'$ and $x'\to^rx$. It is indeed an equivalence relation and the equivalence class $[x]^r$ of $x$ up to $\sim^r$ is called a {\em right} {\em (transition) class}. Set $\llbracket y\rrbracket^r=\{[x]^r\mid x\in\pi^{-1}(y)\}$ and $d_\pi^r(y)=|\llbracket y\rrbracket^r|$ for $y$ in $Y$. The pigeonhole principle guarantees that $d_\pi^r(y)$ is always finite (see Theorem 4.9 of \cite{AllQ13}). The {\em right class degree} $d_\pi^r$ of $\pi$ is defined to be $d_\pi^r=\min_{y\in Y}d_\pi^r(y)$.
  \end{definition}

  If $\pi$ is finite-to-one, then $d_\pi^r$ equals the {\em degree} of $\pi$, which is the minimal cardinality of the set of preimages of a point in $Y$. In some cases, the term {\em class} will be omitted, which are justified as class degree generalises degree in a natural way \cite{AllHJ13,AllQ13}. In particular, the following result was established in \cite{AllHJ13,AllQ13}.
  \begin{theorem}\label{thm:previous_results}
    Let $\pi$ be a factor code from an irreducible shift of finite type $X$ onto a sofic shift $Y$. Let $y$ be a right transitive point of $Y$. Then $d_\pi^r(y)=d_\pi^r$. Also $[x]^r$ contains a right transitive point of $X$ for all $x$ in $\pi^{-1}(y)$.
  \end{theorem}

  We may consider a {\em left} $m$-bridge from $x$ to $x'$ which is a point $\vec x$ in $X$ such that for some $n<m$ we have
  \[ x'|_{(-\infty,n]}=\vec x|_{(-\infty,n]},\vec x|_{[m,\infty)}=x|_{[m,\infty)}\text{ and }\pi(x')=\pi(x)=\pi(\vec x). \]
  Subsequently we consider a {\em left} {\em transition} $x\to^lx'$ from $x$ to $x'$, a sequence $\{\vec x^{(m)}\}_{m\in\mathbb Z}$ of left $m$-bridges from $x$ to $x'$. Then {\em left} {\em equivalence} $x\sim^lx'$ and the {\em left} {\em class}  $[x]^l$ of $x$ up to $\sim^l$ are considered as well. If for $y$ in $Y$ we put $\llbracket y\rrbracket^l=\{[x]^l\mid x\in\pi^{-1}(y)\}$ and $d_\pi^l(y)=|\llbracket y\rrbracket^l|$, then the {\em left class degree} $d_\pi^l=\min_{y\in Y}d_\pi^l(y)$ of $\pi$ is equal to $d_\pi^r$ \cite[\S 6]{AllHJ13}. Hence we may omit ``left'' or ``right'' in front of {\em class degree} and just denote it by $d_\pi$. However, a point in $Y$ may have distinct numbers of right and left classes, so we will continue to distinguish $d_\pi^r(\cdot)$ from $d_\pi^l(\cdot)$.
  
  The next notion which deals with words in a local scope is helpful for many arguments.
  
  \begin{definition}\label{defn:routable_blocks}
    A word $u$ in $\mathcal{B}_l(X)$ is said to be {\em routable} through $M$ at $n$ for some $1<n<l$ and $M\subset\mathcal{A}$ if there is $v$ in $\mathcal{B}(X)$ with $u|_1=v|_1,u|_l=v|_l,\pi(u)=\pi(v)$ and $v|_n\in M$. A word $w$ in $\mathcal{B}(Y)$ is said to be {\em fibre-routable} through $M$ at $n$ if all the words in $\pi^{-1}(w)$ are routable through $M$ at $n$. Define the {\em depth} $d_\pi(w)$ of $w$ by
    \[ d_\pi(w)=\min_{1<n<l}\min\{|M|\mid w\text{ is fibre-routable through }M\text{ at }n\}. \]
  \end{definition}
  
  If two blocks of the same image under $\pi$ are routable through a common single symbol at the same coordinate, then a 2-way bridge is naturally found there. This leads us to another notion concerning blocks.

  \begin{definition}\label{defn:tangled_blocks}
    Given $w$ in $\mathcal{B}(Y)$, a subset $\mathcal{W}$ of $\pi^{-1}(w)$ is said to be {\em tangled} if between any two words of $\mathcal{W}$ lies a 2-way bridge. A partition of $\pi^{-1}(w)$ each member of which is tangled is said to be a {\em tangled} partition. The {\em t-depth} $\tau_\pi(w)$ of $w$ is the smallest cardinality of a tangled partition of $\pi^{-1}(w)$.
  \end{definition}
  
  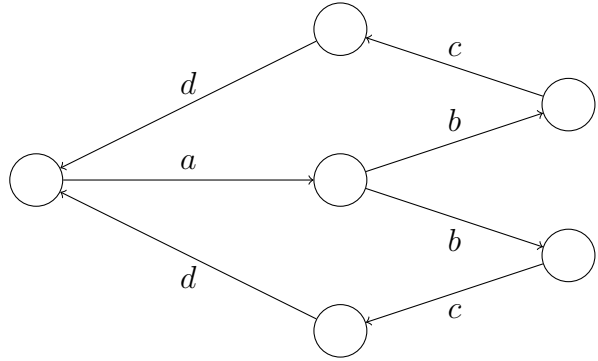
\begin{figure}
    \begin{tikzpicture}[->]
      \node[circle,draw,inner sep=7pt] (I0) at (-2,1) {};
      \node[circle,draw,inner sep=7pt] (I) at (2,1) {};
      \node[circle,draw,inner sep=7pt] (J1) at (5,2) {};
      \node[circle,draw,inner sep=7pt] (J2) at (5,0) {};
      \node[circle,draw,inner sep=7pt] (K1) at (2,3) {};
      \node[circle,draw,inner sep=7pt] (K2) at (2,-1) {};
      
      \draw (I0) to node[above] {$a$} (I);
      \draw (I) to node[above] {$b$} (J1);
      \draw (I) to node[below] {$b$} (J2);
      \draw (J1) to node[above] {$c$} (K1);
      \draw (J2) to node[below] {$c$} (K2);
      \draw (K1) to node[above] {$d$} (I0);
      \draw (K2) to node[below] {$d$} (I0);
    \end{tikzpicture}
    \caption{$\tau_\pi(abcd)=1$ while $d_\pi(abcd)=2$}\label{fig:tangled_and_routable_blocks}
  \end{figure}
  
  \begin{example}\label{eg:tangled_and_routable_blocks}
    Let $\pi$ be the labelling map given in Figure \ref{fig:tangled_and_routable_blocks}. Then the set of the preimages of $w=abcd$ is tangled so that $\tau_\pi(w)=1$ but $d_\pi(w)=2$.
  \end{example}
  
  Beware that $\pi^{-1}(w)$ may admit two distinct tangled partitions of the smallest cardinality. We do not guarantee that $\pi^{-1}(w)$ admits a unique tangled partition which has the smallest cardinality and is ``canonical'' or ``maximal'' in some sense. Even so, in any case $\tau_\pi(w)$ is defined: for every word $w$ in $\mathcal{B}(Y)$ the set of its preimages clearly admits the tangled partition $\{\{u\}\mid u\in\pi^{-1}(w)\}$ of the singletons.
  
  Every extension $uwv$ of a block $w$ in $\mathcal{B}(Y)$ has t-depth and depth no greater than $w$ has. We give an exposition of the case of $\tau_\pi$: let $w$ be in $\mathcal{B}(Y)$ and $\{P_1,\cdots,P_d\}$ a tangled partition of $\pi^{-1}(w)$ where $d={\tau_\pi(w)}$. Given any $u,v$ with $uwv\in\mathcal{B}(Y)$ set $P'_j=\{\alpha\in\pi^{-1}(uwv)\mid\alpha|_{[|u|+1,|u|+|w|]}\in P_j\}$ for $1\le j\le d$. Then for $1\le j\le d$ each $P'_j|_{[|u|+1,|u|+|w|]}$ is contained in $P_j$ and $P'_j$ is tangled.
  
  At first, only the depth of a word was defined in \cite{AllQ13}. The notion of t-depth is better than that of depth in some cases because compactness arguments are applied to the former more naturally. It was shown in \cite[Theorem~4.24]{AllQ13} that $d_\pi=\min_{w\in\mathcal{B}(Y)}d_\pi(w)$. We show that $d_\pi$ equals $\min_{w\in\mathcal{B}(Y)}\tau_\pi(w)$ as well. Throughout the rest of the section, let $\pi$ be a factor code from a shift of finite type $X$ onto an irreducible sofic shift $Y$. Let $y$ be in $Y$ and $d=d_\pi^r(y)$.
 
  \begin{remark}\label{rmk:t_depth_and_depth}
    Let $w$ be in $\mathcal{B}(Y)$ and $l=|w|$. There are $1<n<l$ and $M\subset\mathcal{A}$ with $|M|=d_\pi(w)$ such that $w$ is fibre-routable through a symbol of $M$ at $n$. Set up any order on $M$ and partition the words in $\pi^{-1}(w)$ up to the least symbols of $M$ through which they are routable at $n$. That is, $u$ in $\pi^{-1}(w)$ is classified into $P_a$ if $a$ is the least symbol of $M$ through which $u$ is routable at $n$. For any $a$ in $M$ and two words $u,v$ in $P_a$ there are $u',v'$ with $u'|_n=v'|_n=a,u|_1=u'|_1,u|_l=u'|_l,v|_1=v'|_1$ and $v|_l=v'|_l$. Then $u'|_{[1,n)}av'|_{(n,l]}$ and $v'|_{[1,n)}au'|_{(n,l]}$ form a 2-way bridge between $u$ and $v$. So each $P_a$ is tangled and since $P_a$'s, $a$ in $M$, partition $\pi^{-1}(w)$, we have that $\tau_\pi(w)\le d_\pi(w)$.  
  \end{remark}
  
  \begin{proposition}\label{prop:finitely_many_smaller_depth_than_degree}
    There are only finitely many occurrences of words with depth (resp. t-depth) smaller than $d$ in $y|_{[0,\infty)}$.
  \end{proposition}
  \begin{proof}
    By Remark~\ref{rmk:t_depth_and_depth} it suffices to show that at most finitely many occurrences of words with t-depth smaller than $d$ are allowed in $y|_{[0,\infty)}$. Let us have an infinite sequence $0<m_1<n_1<m_2<n_2<\cdots$ with $w_k=y|_{[m_k,n_k]},d_k=\tau_\pi(w_k),k\in\mathbb{N}$. Let $\mathcal{P}_k$ be a tangled partition of $\pi^{-1}(w_k)$ of cardinality $d_k$ for every $k\in\mathbb{N}$.
    
    Let $d'=\liminf_kd_k$. Take any preimages $x_j,0\le j\le d'$, of $y$. Let $P_{j,k}$ be the element of $\mathcal{P}_k$ which contains $x_j|_{[m_k,n_k]}$, for $k\in\mathbb{N}$. Given $k\in\mathbb{N}$ with $d_k=d'$ there are $0\le i<j\le d'$ with $P_{i,k}=P_{j,k}$ so that we have a 2-way bridge between $x_i|_{[m_k,n_k]}$ and $x_j|_{[m_k,n_k]}$. So there are $0\le i<j\le d'$ such that we have a 2-way bridge between $x_i|_{[m_k,n_k]}$ and $x_j|_{[m_k,n_k]}$ for infinitely many $k$ with $d_k=d'$, i.e., $x_i\sim^rx_j$. Hence, it is impossible to take more than $d'$ preimages of $y$ which are not right equivalent to each other, and $d\le d'$.
  \end{proof}
  
  A {\em recurrent point} is a point any word of which occurs infinitely often to the right. Corollary~\ref{cor:min_depths_and_degree} is immediate from Proposition~\ref{prop:finitely_many_smaller_depth_than_degree}.
  
  \begin{corollary}\label{cor:min_depths_and_degree}
    If $y$ is recurrent, then
    \[ d=\min\{d_\pi(y|_{[m,n]})\mid m\le n\}=\min\{\tau_\pi(y|_{[m,n]})\mid m\le n\}. \]
    Furthermore, if $y$ is right transitive, then
    \[ d_\pi=d=\min\{d_\pi(y|_{[m,n]})\mid m\le n\}=\min\{\tau_\pi(y|_{[m,n]})\mid m\le n\}. \]
  \end{corollary}

\section{Class-closing factor codes}\label{sec:closing_factors}
  We introduce the notion of class-closing property. Its equivalent conditions and several properties are found. The main result in the present section is the following theorem.
    
  \begin{theorem}\label{thm:class-closing_is_continuing}
    Let $\pi$ be a factor code from an irreducible shift of finite type $X$ onto another shift of finite type $Y$. If $\pi$ is right class-closing, then it is right continuing.
  \end{theorem}

  We start with a definition and an example.

  \begin{definition}\label{defn:class-closing}
    Let $\pi$ be a factor code from an irreducible shift of finite type $X$ onto a sofic shift $Y$. If any two left (resp., right) asymptotic points $x$ and $x'$ in $X$ with $\pi(x)=\pi(x')$ are right (resp., left) equivalent, then $\pi$ is said to be {\em right} (resp., {\em left}) {\em class-closing}. If $\pi$ is right and left class-closing at the same time, then it is said to be {\em bi-class-closing}.
  \end{definition}
  
  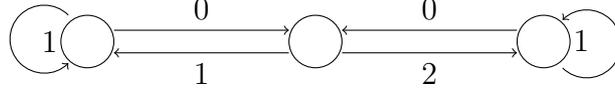
\begin{figure}
    \begin{tikzpicture}[->]
      \node[circle,draw,inner sep=7pt] (I) at (0,1) {};
      \node[circle,draw,inner sep=7pt] (J) at (3,1) {};
      \node[circle,draw,inner sep=7pt] (K) at (6,1) {};
      
      \draw (0.35,1.15) to (1.5,1.15) node[above] {$0$} to (2.65,1.15) ;
      \draw (2.65,0.85) to (1.5,0.85) node[below] {$1$} to (0.35,0.85) ;
      \draw (5.65,1.15) to (4.5,1.15) node[above] {$0$} to (3.35,1.15) ;
      \draw (3.35,0.85) to (4.5,0.85) node[below] {$2$} to (5.65,0.85) ;
      \draw (6.25,0.65) arc (225:495:0.45) node[below right] {$1$};
      \draw (-0.25,1.35) arc (45:315:0.45) node[above left] {$1$};
    \end{tikzpicture}
    \caption{Right class-closing but not left class-closing factor code}\label{fig:class-closing}
  \end{figure}
  
  Note that two left asymptotic points are automatically left equivalent so that we only need to consider their right equivalence. The class-closing condition is slightly weaker than closing one in that the latter require two asymptotic points with the same image to be actually equal. Any right closing factor code which is not left closing is also right class-closing but not left class-closing. For a concrete example, see Example~\ref{eg:class-closing}.
  
  \begin{example}\label{eg:class-closing}
    The factor code given in Figure~\ref{fig:class-closing} as the labelling map on the edges is not left class-closing since it has two right asymptotic preimages of ${}^\infty101^\infty$ which are not left equivalent. It is, however, right resolving.
  \end{example}

  We find some equivalent conditions for class-closing property. Let $G=(\mathcal{V},\mathcal{E})$ be an irreducible graph. Let $X=\mathsf{X}_G$ and $Y=\pi(X)$ where $\pi$ is a labelling map on $\mathcal{E}$. For $\mathcal{I}\subset\mathcal{V}$ and $v$ in $\mathcal{B}(Y)$, let $\mathsf{t}_{G,\mathcal{I}}(v)=\mathsf{t}_G\{w\in\pi^{-1}(v)\mid\mathsf{i}_G(w)\in\mathcal{I}\}$. A subset $\mathcal{J}$ of $\mathcal{V}$ is said to be  {\em right accessible} from a vertex $I$ in $\mathcal{V}$ via a word $v$ in $\mathcal{B}(Y)$ if $\mathcal{J}\subset\mathsf{t}_{G,\{I\}}(v)$. A set $\Gamma$ of paths with the same label over $G$ is said to be  {\em right accessible} from $I$ if $\mathsf{i}_G(\Gamma)$ is right accessible from $I$. Symmetrically we define {\em left accessible} vertices and paths. Accessible sets are based on a similar idea as {\em compatible} sets which Nasu introduced to encode bi-closing factor codes between shifts of finite type to bi-resolving ones \cite{Nas83}.
  
  \begin{theorem}\label{thm:class-closing_conditions}
    Let $G=(\mathcal{V},\mathcal{E})$ be an irreducible graph. Let $X=\mathsf{X}_G$ and $Y=\pi(X)$ where $\pi$ is a labelling map on $\mathcal{E}$. Then the following are equivalent:
    \begin{enumerate}
      \item\label{itm:class-closing} $\pi$ is right class-closing.
      \item\label{itm:class_mutual_separation} Any two points from distinct right classes over the same point in $Y$ are mutually separated.
      \item\label{itm:left_class_in_right_class} $[x]^l\subset[x]^r$ for all $x$ in $X$. Equivalently, any right class is a union of left classes.
      \item\label{itm:vertex_delay} We have an integer $D\ge0$ such that the paths of length $D+1$ outgoing from any single vertex of $G$ are tangled if they have the same label.
      \item\label{itm:access_delay} We have an integer $D\ge0$ such that the paths of length $D+1$ right accessible from any single vertex of $G$ are tangled if they have the same label.
    \end{enumerate}
    Furthermore, $D$ found in (\ref{itm:vertex_delay}) satisfies the role of $D$ in (\ref{itm:access_delay}) and the vice versa.
  \end{theorem}
  \begin{proof}
    (\ref{itm:access_delay})$\implies$(\ref{itm:class_mutual_separation}): Assume (\ref{itm:access_delay}). It suffices to show that if two points with the same image of $X$ meet each other at least once then they are right equivalent. Let two points $x$ and $x'$ in $X$ meet each other and have the same image $y$. We may assume that $x|_0=x'|_0=e$ without loss of generality. For any $n>0$ two paths $x|_{[n,n+D]}$ and $x'|_{[n,n+D]}$ are right accessible from $\mathsf{t}_G(e)$ via $y|_{[1,n-1]}$. By (\ref{itm:access_delay}) we have a 2-way right transition between $x$ and $x'$ which is just a sequence of 2-way bridges between $x|_{[n,n+D]}$ and $x'|_{[n,n+D]}$.
    
    (\ref{itm:class_mutual_separation})$\implies$(\ref{itm:left_class_in_right_class}): Assume that $[x]^l\setminus[x]^r$ is nonempty and contains a point $x'$ for some $x$ in $X$. We have $n<0$ and a symbol $a$ in $\mathcal{A}$ such that $x$ and $x'$ are routable through $a$ at $n$. Taking routed points if necessary, we may assume that $x|_n=x'|_n=a$. That is, $x$ and $x'$ are not mutually separated with $x'\not\in[x]^r$.
    
    (\ref{itm:left_class_in_right_class})$\implies$(\ref{itm:class-closing}): Suppose that $\pi$ is not right class-closing. Then there are two left asymptotic points $x$ and $x'$ with $\pi(x)=\pi(x')$ but which are not right equivalent. However they are left equivalent as are left asymptotic. That is, $x'\in[x]^l\setminus[ x]^r$ and $[x]^l\not\subset[ x]^r$.
    
    (\ref{itm:class-closing})$\implies$(\ref{itm:vertex_delay}): Suppose that (\ref{itm:vertex_delay}) does not hold. Then for any $D\in\mathbb{N}$ there is a vertex $I$ in $\mathcal{V}$ such that we can find paths $\alpha_D$ and $\beta_D$ of length $D$ outgoing from $I$ with $\pi(\alpha_D)=\pi(\beta_D)$ but without a 2-way bridge between them. Because $\mathcal{V}$ is finite there is a vertex $I$ such that for all $D$ we can find such outgoing paths $\alpha_D$ and $\beta_D$. Without loss of generality we may say that for infinitely many $D\in\mathbb{N}$ there is no bridge from $\alpha_D$ to $\beta_D$. By compactness of $X$ there are two right infinite paths $\alpha$ and $\beta$ outgoing from $I$ such that $\pi(\alpha)=\pi(\beta)$ and there is no bridge from $\alpha$ to $\beta$. Take a left infinite path $\lambda$ of $X$ terminating at $I$ and set two left asymptotic points $x=\lambda\alpha$ and $x'=\lambda\beta$. Then $x\not\to^rx'$. Hence $\pi$ is not right class-closing.
    
    (\ref{itm:vertex_delay})$\implies$(\ref{itm:access_delay}): Find $D$ satisfying (\ref{itm:vertex_delay}). We use the induction. The first step is for any two paths of length $D+1$ over $G$ which are labeled the same and are right accessible from a vertex of $G$ via the empty word $\varepsilon$. The case is just a restatement of (\ref{itm:vertex_delay}) in terms of accessibility.
    
    Next, assume that (\ref{itm:access_delay}) holds for any two paths of length $D+1$ over $G$, both labeled the same and right accessible from a vertex of $G$ via a word in $\mathcal{B}_n(Y)$ for some $n\in\mathbb{Z}^+$. We will show that the same holds for the paths right accessible via a word in $\mathcal{B}_{n+1}(Y)$.
    
    Choose over $G$ two paths $\alpha$ and $\beta$ of length $D+1$ labeled $w$ and right accessible from $I$ in $\mathcal{V}$ via $u$ in $\mathcal{B}_{n+1}(Y)$. There are paths $\gamma$ and $\kappa$ labeled $u$ from $I$ to $\mathsf{i}_G(\alpha)$ and to $\mathsf{i}_G(\beta)$, respectively, over $G$. Put $\alpha'=\gamma|_{n+1}\alpha|_{[1,D]}$ and $\beta'=\kappa| _{n+1}\beta|_{[1,D]}$. Since $\alpha'$ and $\beta'$ are labeled the same word $u|_{n+1}w|_{[1,D]}$ and right accessible from $I$ via $u|_{[1,n]}$, the induction assumption gives a bridge $\vec\alpha'$ from $\alpha'$ to $\beta'$. The vertex $\mathsf{i}_G(\alpha)=\mathsf{t}_G(\alpha'|_1)=\mathsf{t}_G(\vec\alpha'|_1)=\mathsf{i}_G(\vec\alpha'|_2)$ has two outgoing paths $\alpha$ and $\vec\alpha'|_{[2,D]}\beta|_{[D,D+1]}$ of the same label $w$. Applying (\ref{itm:vertex_delay}) to them we obtain a bridge from $\alpha$ to $\vec\alpha'|_{[2,D]}\beta|_{[D,D+1]}$, that is, a path in $\pi^{-1}(w)$ starting with $\alpha|_1$ and ending with $\beta|_{D+1}$. The path is a bridge from $\alpha$ to $\beta$ as well. Similarly we have a bridge from $\beta$ to $\alpha$. The proof is complete.
  \end{proof}
  
  \begin{remark}\label{rmk:class_separation}
    The mutual separation in (\ref{itm:class_mutual_separation}) is not a conjugacy invariant. Even if $x$ and $x'$ in $X$ are mutually separated and $\phi$ is a conjugacy from $X$, it may happen that $\phi(x)$ and $\phi(x')$ are not mutually separated. To address this flaw, let $d$ be a metric on $X$ which gives the same product topology as before. It does exist in any case \cite{LinM95}. Now, if there is $\delta>0$ such that $d(x,x')>\delta$ for any two points $x,x'$ in the distinct right (resp., left) classes, then $\pi$ is said to {\em separate right} (resp., {\em left}) {\em classes}. This notion turns out to be invariant under conjugacy and, together with (\ref{itm:class_mutual_separation}) and (\ref{itm:left_class_in_right_class}), leads to Corollary~\ref{cor:separate_right_and_left_classes}.
  \end{remark}
  
  \begin{corollary}\label{cor:separate_right_and_left_classes}
    Let $\pi$ be a factor code from an irreducible shift of finite type $X$ onto a sofic shift $Y$. Then $\pi$ is right class-closing if and only if $\pi$ separates right classes if and only if $[x]^l\subset[x]^r$ for all $x$ in $X$. Also $\pi$ is bi-class-closing if and only if $[x]^l=[x]^r$ for all $x$ in $X$.
  \end{corollary}
  
  The smallest $D$ satisfying (\ref{itm:vertex_delay}) or (\ref{itm:access_delay}) in Theorem~\ref{thm:class-closing_conditions} is called the {\em delay} of $\pi$. When $
  \pi$ is a finite-to-one closing factor code from an irreducible shift of finite type, its class-closing delay equals its {\em closing delay} \cite[\S5.1]{LinM95}. If $\pi$ has delay $0$ then it is right resolving. Also note that for any factor code which is not necessarily right class-closing, the condition (\ref{itm:class_mutual_separation}) holds for right transitive points, which are ``typical'' points, of $Y$ \cite{AllHJ13}.

  We recall another generalisation of closing property given in \cite{BoyT84}, called {\em continuing} property.

  \begin{definition}\label{defn:continuing}
    Let $\pi$ be a factor code from a shift space $X$ onto $Y$. If for any points $x$ in $X$ and $y$ in $Y$ left (resp., right) asymptotic to $\pi(x)$ there is a preimage of $y$ left (resp., right) asymptotic to $x$, then $\pi$ is said to be {\em right} (resp., {\em left}) {\em continuing}. If $\pi$ is right and left continuing at the same time, then it is said to be {\em bi-continuing}.
  \end{definition}
  
  The property is invariant under conjugacy. Theorem~\ref{thm:class-closing_is_continuing} generalises a well-known fact that a right closing finite-to-one factor code between irreducible shifts of finite type is right continuing. The fact itself is used in the proof of the result.
  
  For the proof of the theorem, we construct $\bar{G}$ from $G$ and a labelling map $\bar\pi$ on it using the {\em subset construction} as follows: the vertices of $\bar{G}$ are set to be the nonempty subsets of $\mathcal{V}$. Let $\bar{I}$ and $\bar{J}$ be in $\bar{\mathcal{V}}=\mathcal{V}(\bar{G})$. Put an edge $\bar{e}$ from $\bar{I}$ to $\bar{J}$ and set $\bar\pi(\bar{e})=a$ if $\bar{J}=\mathsf{t}_{G,\bar{I}}(a)$. Name the new labelling map on $\bar{G}$ as $\bar\pi$. Then $\bar\pi$ is right resolving and $\bar\pi(\mathsf{X}_{\bar{G}})=\pi(\mathsf{X}_G)=Y$.
  
  The notion of {\em sink} makes things simpler in the proof. An irreducible subgraph $H$ of $\bar{G}$ is a {\em sink} if any outgoing edge from a vertex of $H$ is an edge of $H$ and its terminal vertex again belongs to $H$. Choose any $I$ in $\mathcal{V}$ and let $H$ be a sink of $\bar{G}$ which can be reached from $\{I\}$ in $\bar{\mathcal{V}}$ over $\bar{G}$. It exists, is irreducible and $\bar\pi(\mathsf{X}_H)=Y$, always. Note that for any $\bar{I}$ in $\mathcal{V}(H)$ there is $u$ in $\mathcal{B}(G)$ with $\bar{I}=\mathsf{t}_{G,\{I\}}(u)$ and that a path from $\bar{I}$ to $\bar{J}$ in $\mathcal{V}(H)$ over $\bar{G}$ may be regarded as the set of paths from a subset of $\bar{I}\subset\mathcal{V}$ onto $\bar{J}\subset\mathcal{V}$ over $G$. Refer to \cite[\S3.3,\S4.4]{LinM95} for detailed expositions about the subset construction and sink.

  \begin{figure}
    \begin{tikzpicture}[->,to path={.. controls +(1,0) and +(-0.7,0) .. (\tikztotarget) \tikztonodes}]
      \node[coordinate,label=$I$] (I) at (0,1) {};
      \node[coordinate] (I0) at (0,2) {};
      \node[coordinate] (I1) at (0,3) {};
      \node[ellipse,fit=(I) (I0) (I1),draw,inner sep=8pt,label=$\bar{I}$] {};
      
      \node[coordinate,label=$J$] (J) at (3,1.5) {};
      \node[coordinate] (J0) at (3,3) {};
      \draw (I) to node[above] {$v$} (J);
      \draw (I0) to node[above] {\small{$v'\in\pi^{-1}(\pi(v))$}} (J0);
      \node[ellipse,fit=(J) (J0),draw,inner sep=8pt] {};
      
      \node[coordinate,label=below right:$\mathsf{t}_G(x_m)$] (L) at (5,1) {};
      \node[coordinate] (L0) at (5,2.5) {};
      \draw (J) to node[above] {$x_m$} (L);
      \node[ellipse,fit=(L) (L0),draw,inner sep=8pt,label=$\mathsf{t}_{\bar{G}}(\bar{u}|_{|v|+1})$] {};
      
      \node[coordinate,label=below right:$\mathsf{t}_G(x_n)$] (M) at (9.5,1.5) {};
      \node[coordinate] (M0) at (9.5,3) {};
      \draw[dashed] (L) to node[above] {$x|_{[m+1,n-1]}$} (M);
      \node[ellipse,fit=(M) (M0),draw,inner sep=8pt,label=$\mathsf{t}_{\bar{G}}(\bar{u}|_{|v|+n-m})$] {};

      \node[coordinate] (K) at (12,2) {};
      \node[coordinate] (K1) at (12,3) {};
      \draw (M) to node[above] {$x_n$} (K);
      \node[ellipse,fit=(K) (K1),draw,inner sep=8pt,label=$\bar{K}$] {};
    \end{tikzpicture}
    \caption{Finding $\bar{u}$}\label{fig:subset_graph_right_continuing}
  \end{figure}
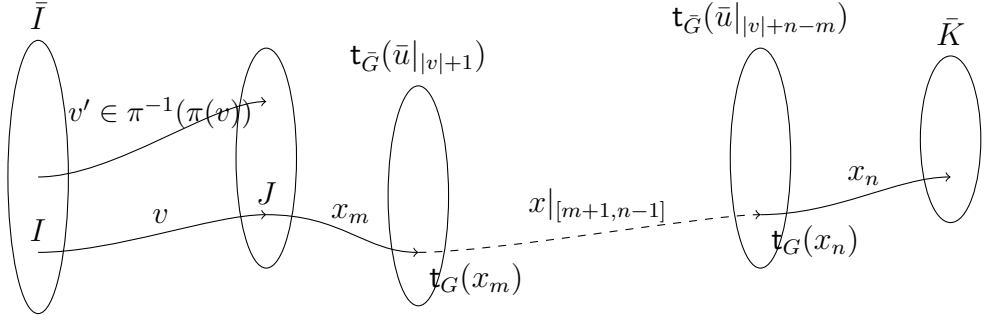    

  \begin{proof}[Proof of Theorem~\ref{thm:class-closing_is_continuing}]\label{proof:class-closing_is_continuing}
    Reset $\bar{G}$ to be the sink $H$ defined as above for the convenience. Then $\bar\pi$ is right continuing as a right resolving factor code onto a shift of finite type.
    
    We claim that for every point $x$ in $X$ there is a point $\bar{x}$ in $\bar{X}=\mathsf{X}_{\bar{G}}$ such that each $\bar\pi(\bar{x})=\pi(x)$ and $\mathsf{t}_{\bar{G}}(\bar{x}|_i)$ contains $\mathsf{t}_G(x|_i)$ for all $i\in\mathbb{Z}$. To show the claim, take any $\bar{I}$ in $\bar{\mathcal{V}}$ and a vertex $I$ of $G$ contained in $\bar{I}$. Given $m\le n$ let $v$ be a path from $I$ to $J=\mathsf{i}_G(x|_{[m,n]})$ over $G$ and let $\alpha=\pi(vx|_{[m,n]})$. Over $\bar{G}$ there appears a path $\bar{u}$ labeled $\alpha$ from $\bar{I}$ into a vertex $\bar{K}=\{\mathsf{t}_G(u)\mid\mathsf{i}_G(u)\in\bar{I},\pi(u)=\alpha\}$ of $\bar{G}$. Clearly $\mathsf{t}_{\bar{G}}(\bar{u}|_{|v|+i})$ contains $\mathsf{t}_G(x|_{m-1+i})$ for each $i=1,\cdots,n-m+1$. At last, by compactness we can find $\bar{x}$ as well (see Figure~\ref{fig:subset_graph_right_continuing}).
    
    Now given $x$ in $X$ and $y$ in $Y$ with $\pi(x)$ and $y$ left asymptotic, first fix $\bar{x}$ found by the claim. Since $\bar\pi$ is right continuing and $\bar\pi(\bar{x})=\pi(x)$, there exists a point $\bar{x}'$ in $\bar\pi^{-1}(y)$ left asymptotic to $\bar{x}$. Without loss of generality, assume that $\bar{x}|_{(-\infty,0]}=\bar{x}'|_{(-\infty,0]}$. It is easy to detach from $\bar{x}'|_{[n,\infty)}$ a right infinite path  $x'_n|_{[n,\infty)}$ in $X$ such that $\pi(x'_n|_{[n,\infty)})=y|_{[n,\infty)}$ and $\mathsf{t}_{\bar{G}}(\bar{x}'|_i)$ contains $\mathsf{t}_G(x'_n|_i)$ for all $n\le0$ and $i\ge n$. By compactness of $X$ we have a point $x'$ in $X$ such that $\pi(x')=y$ and $\mathsf{t}_{\bar{G}}(\bar{x}'|_i)$ contains $\mathsf{t}_G(x'|_i)$ for all $i$ in $\mathbb{Z}$. Let $J=\mathsf{t}_G(x'|_0)$.
    
    Let $D$ be the delay of $\pi$. Regard $\bar{x}|_{[-D,0]}$ as a set of paths of length $D+1$ over $G$ from a subset of $\bar{I}=\mathsf{t}_G(\bar{x}'|_{-D-1})$ into all the vertices in $\bar{J}=\mathsf{t}_G(\bar{x}|_0)$. Since $\bar{I}$ is right accessible from a vertex in $\mathcal{V}$ via some path, say $u$, over $G$, the set $\bar{x}|_{[-D,0]}$ is also an right accessible set of paths over $G$. As $\pi$ is right class-closing, it is tangled. Let $w$ be a bridge from $x|_{[-D,0]}$ to $x'|_{[-D,0]}$. Then $x|_{(-\infty,-D-1]}wx'|_{[1,\infty]}$ is left asymptotic to $x$ and is sent to $y$, completing the proof.
  \end{proof}
  
  Since a continuing factor of a shift of finite type is of finite type \cite{Yoo13}, a right class-closing code from an irreducible shift of finite type is right continuing if and only if its image is of finite type. Trivially, every right closing factor code from an irreducible shift of finite type onto a strictly sofic shift is not right continuing. For an example, see Example~\ref{eg:class-closing_but_not_continuing}.
  
  \begin{example}\label{eg:class-closing_but_not_continuing}
    Let $\pi$ be the labelling given in Figure~\ref{fig:class-closing_but_not_continuing}. Then it is bi-class-closing with both delays 2. Any positive power $(ab)^k$ of $ab$ is not synchronizing as $bb(ab)^k$ and $(ab)^kb$ are allowed while $bb(ab)^kb$ is not, so the image shift space is not of finite type. Hence $\pi$ is neither right nor left continuing.
  \end{example}

  \begin{figure}
    \begin{tikzpicture}[->]
      \node[circle,draw,inner sep=7pt] (I) at (0,1) {};
      \node[circle,draw,inner sep=7pt] (J) at (3,1) {};
      \node[circle,draw,inner sep=7pt] (K) at (6,1) {};
      
      \draw (3.35,1.15) to (4.5,1.15) node[above] {$a$} to (5.65,1.15);
      \draw (5.65,0.9) to (4.5,0.9) node[below] {$b$} to (3.35,0.9);
      \draw (0.35,1.15) to (1.5,1.15) node[above] {$a$} to (2.65,1.15);
      \draw (2.65,0.9) to (1.5,0.9) node[below] {$b$} to (0.35,0.9);
      \draw (3.1,1.4) arc (-75:255:0.4) node[above] {$a$};
    \end{tikzpicture}
    \caption{Bi-class-closing but neither right nor left continuing factor code}\label{fig:class-closing_but_not_continuing}
  \end{figure}
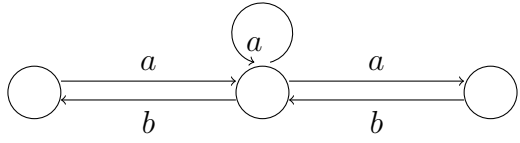

  Before moving on to the next section, we note a difference between closing factor codes and class-closing factor codes. We present an example of a bi-continuing factor code between irreducible shifts of finite type which is neither left nor right class-closing. The example contrasts with finite-to-one right (resp., left) continuing factor codes between irreducible shifts of finite type, which are automatically right (resp., left) closing \cite{BoyT84}.
  
  \begin{example}\label{eg:continuing_but_not_class-closing}
    Let $\pi$ be the labelling given in Figure~\ref{fig:continuing_but_not_class-closing}. Then it is bi-continuing but neither left nor right class-closing: ${}^\infty01^\infty$ has two left asymptotic preimages which are not right equivalent and ${}^\infty10^\infty$ has two right asymptotic preimages which are not left equivalent.
  \end{example}

  \begin{figure}
    \begin{tikzpicture}[->]
      \node[circle,draw,inner sep=7pt] (I) at (0,1) {};
      \node[circle,draw,inner sep=7pt] (J) at (3,1) {};
      
      \draw (0.35,1.15) to (1.5,1.15) node[above] {$0$} to (2.65,1.15);
      \draw (2.65,0.9) to (1.5,0.9) node[below] {$0$} to (0.35,0.9);
      \draw (3.25,0.65) arc (225:495:0.45) node[above] {$1$};
      \draw (-0.2,1.45) arc (0:300:0.4) node[above] {$0$};
      \draw (-0.2,0.55) arc (360:60:0.4) node[below] {$1$};
    \end{tikzpicture}
    \caption{Bi-continuing but neither right nor left class-closing factor code}\label{fig:continuing_but_not_class-closing}
  \end{figure}
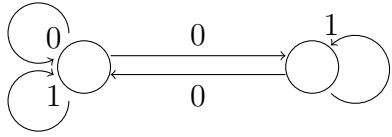

\section{Constant-class-to-one factor codes}\label{sec:c-to-one_factors}
  
  We establish relations between class-closing factor codes, continuing factor codes and constant-class-to-one factor codes.
  
  \begin{definition}\label{defn:constant-class-to-one}
    Let $\pi$ be a factor code from a shift of finite type $X$ onto a sofic shift $Y$. We say that $\pi$ is {\em constant-class-to-one} if $d_\pi^r(y)=d_\pi$ for all $y$ in $Y$. 
  \end{definition}
  
  \begin{figure}
    \begin{tikzpicture}[->]
      \node[circle,draw,inner sep=7pt] (I) at (0,1) {};
      \node[circle,draw,inner sep=7pt] (J) at (3,1) {};
      \node[circle,draw,inner sep=7pt] (K) at (0,3.5) {};
      
      \draw (-0.15,1.35) to (-0.15,2.25) node[left] {$b$} to (-0.15,3.15);
      \draw[<-]  (0.15,1.35) to (0.15,2.25) node[right] {$c$} to (0.15,3.15);
      \draw (0.35,1.15) to (1.5,1.15) node[above] {$a$} to (2.65,1.15);
      \draw[<-] (0.35,0.9) to (1.5,0.9) node[below] {$b$} to (2.65,0.9);
      \draw (-0.25,1.3) arc (15:300:0.4) node[above] {$a$};
      \draw (-0.2,0.65) arc (120:420:0.4) node[below] {$c$};
    \end{tikzpicture}
    \caption{A fibre-mixing factor code: every bi-infinite labelled paths has exactly one transition class}\label{fig:constant-class-to-one}
  \end{figure}
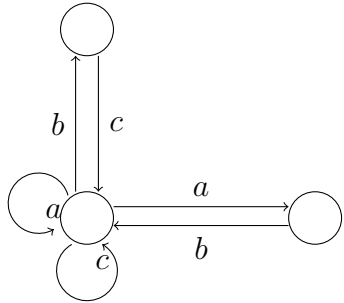
  
  Later in this section it will be shown that {\em left} or {\em right} does not matter in Definition~\ref{defn:constant-class-to-one}. If $\pi$ has constant preimages over every image point, then it is called {\em constant-to-one}. Later, we can see that $\pi$ is constant-to-one if and only if it is finite-to-one and constant-class-to-one. When $\pi$ is constant-class-to-one with $d_\pi=1$, it is called {\em fibre-mixing}. The property was defined and studied in \cite{Yoo10}. An example of such a factor code is given in Figure~\ref{fig:constant-class-to-one}.
  
  It has been known that a finite-to-one factor between irreducible shifts of finite type is open if and only if it is constant-to-one if and only if it is bi-closing \cite{CovP77,Nas83}. We also know that a factor code between irreducible shifts of finite type is open if and only if it is bi-continuing \cite{Jun11}. We aim to extend some of these equivalences to infinite-to-one cases. First, with the help of Lemma~\ref{lmm:separated_classes}, a constant-class-to-one factor code is shown to be bi-class-closing and bi-continuing.
    
  \begin{lemma}\label{lmm:separated_classes}
    Let $\pi$ be a factor code from an irreducible shift of finite type $X$ onto a sofic shift $Y$. There is $\delta>0$ such that for any $x$ in $X$ the set $\llbracket\pi(x)\rrbracket^r$ contains $d=d_\pi$ right classes including $[x]^r$ the distances among which are no less than $\delta$.
  \end{lemma}
  \begin{proof}
    Without loss of generality assume $X$ to be 1-step and $\pi$ to be 1-block. We find such $d$ classes which are mutually separated. Let $y=\pi(x)$ and let $\{x_n\}_{n\in\mathbb{N}}$ be a sequence of right transitive points with $x_n|_{(-\infty,n]}=x|_{(-\infty,n]}$. Then $y_n=\pi(x_n)$ is right transitive and $y_n|_{(-\infty,n]}=y|_{(-\infty,n]}$ for every $n\in\mathbb{N}$.
    
    For every $n$ in $\mathbb{N}$ there are $d$ mutually separated right classes $C_{n,1},\cdots,C_{n,d}$ over $y_n$ such that $x_n^{(1)}=x_n$ is in $C_{n,1}$ since each $y_n$ is right transitive \cite[Theorem 4.4]{AllHJ13}. Take $x_n^{(j)}$ from $C_{n,j}$ for each $2\le j\le d$ and $n\in\mathbb{N}$. There is an increasing sequence $\{n_k\}_{k\in\mathbb{N}}$ such that $x_{n_k}^{(j)}$ converges to some $x^{(j)}$ as $k\to\infty$ for all $1\le j\le d$, $x_{n_k}^{(j)}|_{[-k,k]}=x^{(j)}|_{[-k,k]}$ and  $x^{(1)}=x$. Since $X$ is 1-step, by changing the left infinite tails we may assume that $x_{n_k}^{(j)}|_{(-\infty,k]}=x^{(j)}|_{(-\infty,k]}$ for all $1\le j\le d$ and $k\in\mathbb{N}$. It is clear that $\pi(x^{(j)})=y$ for $1\le j\le d$.
    
    Fix any $i\ne j$. We claim that there is no right bridge from $x^{(i)}$ to $x^{(j)}$. For a contradiction, assume that there is $\vec{u}$ in $\mathcal{B}_N(X)$ such that $x^{(i)}|_{(-\infty,0]}\vec{u}x^{(j)}|_{[N+1,\infty)}$ is in $\pi^{-1}(y)$ for some $N\in\mathbb{N}$. It follows that for $k\geq N+1$ we have $x^{(i)}_{n_k}|_{(-\infty,0]}\vec{u}x^{(j)}_{n_k}|_{[N+1,\infty)}\in\pi^{-1}(y_{n_k})$, which contradict the assumption that $C_{n_k,1},\cdots,C_{n_k,d}$ are mutually separated. 
   
    Now we show that $C_i=[x^{(i)}]^r$ and $C_j=[x^{(j)}]^r$ are mutually separated. Take arbitrary $z^{(i)}$ and $z^{(j)}$ in $C_i$ and $C_j$, respectively. For each $m\in\mathbb{N}$ find a right $m$-bridge $z^{(i)}_m$ from $z^{(i)}$ to $x^{(i)}$ such that $z^{(i)}_m|_{[m',\infty)}=x^{(i)}|_{[m',\infty)}$ for some $m'>m$. For any $k>m'$ the point $\bar{z}^{(i)}_m=z^{(i)}_m|_{(-\infty,m']}x^{(i)}_{n_k}|_{[m',\infty)}$ is legal in $X$, is mapped to $y_{n_k}$, and $\lim_m\bar{z}_m^{(i)}=\lim_mz^{(i)}_m=z^{(i)}$. Simultaneously, we also find preimages $\bar{z}_m^{(j)}$ of $y_{n_k}$ which are right transitive for all $m\in\mathbb{N}$ and converge to $z^{(j)}$.
    
    Note that since $x^{(i)}_{n_k}$ and $x^{(j)}_{n_k}$ are in different classes which are mutually separated, then $\bar{z}_m^{(i)}\sim^rx^{(i)}_{n_k}$ and $\bar{z}_m^{(j)}\sim^rx^{(j)}_{n_k}$ are mutually separated. Therefore their respective convergent points $z^{(i)}$ and $z^{(j)}$ are mutually separated, which implies that $C_i$ and $C_j$ are mutually separated.
  \end{proof}

  \begin{theorem}\label{thm:constant-class-to-one}
    Let $\pi$ be a 1-block factor code from an irreducible 1-step shift of finite type $X$ onto a sofic shift $Y$. The followings are equivalent:
    \begin{enumerate}
      \item\label{itm:constant-class-to-one} $\pi$ is constant-class-to-one with $d_\pi=d$.
      \item\label{itm:constant_depth} There is $N>0$ such that all words in $\mathcal{B}_N(Y)$ have the minimal depth $d$.
      \item\label{itm:constant_t-depth} There is $N>0$ such that all words in $\mathcal{B}_N(Y)$ have the minimal t-depth $d$.
    \end{enumerate}    
    If any of the above three conditions holds then $\pi$ is bi-class-closing and bi-continuing.
  \end{theorem}
  \begin{proof}
    (\ref{itm:constant_t-depth})$\implies$(\ref{itm:constant-class-to-one}): Let for a fixed $N>0$ all words in $\mathcal{B}_N(Y)$ have the minimal t-depth $d$. Given any point $y$ in $Y$ choose any word $w$ in $\mathcal{B}_N(\omega(y))$ where $\omega(y)$ is the $\omega$-limit set of $y$. Then $d_\pi^r(y)\le \tau_\pi(w)=d$ by Proposition ~\ref{prop:finitely_many_smaller_depth_than_degree}. At the same time $d=d_\pi\le d_\pi^r(y)$. Hence $d_\pi^r(y)=d$.
    
    (\ref{itm:constant-class-to-one})$\implies$(\ref{itm:constant_depth}): Suppose that for any $N>0$ there is a word in $\mathcal{B}_N(Y)$ whose depth is greater than $d=d_\pi$. By compactness there is a point $y$ in $Y$ such that all words occurring in it have depth greater than $d$. According to Theorem 4.22 of \cite{AllQ13} always occurs a word with depth $d_\pi^r(y)$ in $y$. Hence $d_\pi^r(y)>d$.
    
    (\ref{itm:constant_depth})$\implies$(\ref{itm:constant_t-depth}): Let for a fixed $N>0$ all words in $\mathcal{B}_N(Y)$ have the minimal depth $d$. Applying Remark~\ref{rmk:t_depth_and_depth} we obtain $\tau_\pi(w)\le d_\pi(w)=d$ for all $w$ in $\mathcal{B}_N(Y)$. Applying Corollary~\ref{cor:min_depths_and_degree} we obtain $d=d_\pi\le\tau_\pi(w)$ for all $w$ in $\mathcal{B}_N(Y)$. Hence $\tau_\pi(w)=d$ for all $w$ in $\mathcal{B}_N(Y)$.
    
    Now assume that $\pi$ is constant-class-to-one with $d=d_\pi$ and that $N$ is such a number that $d_\pi(w)=d$ for all $w$ in $\mathcal{B}_N(Y)$. Using Lemma~\ref{lmm:separated_classes} we can find some $\delta>0$ such that for any point $y$ of $Y$ there are $d$ right classes over $y$ any two of which are separated by some distance greater than $\delta$. Only $d$ right classes exist over $y$ and the existence of $\delta$ immediately implies that $\pi$ separates right classes. Thus $\pi$ is right class-closing. In a similar way $\pi$ is left class-closing and hence is bi-class-closing.
    
    Finally let $x$ be in $X$ and $y$ a point left asymptotic to $\pi(x)$ in $Y$. Assume without loss of generality that $w=y|_{[1,N]}=\pi(x)|_{[1,N]}$. There is a tangled partition $\mathcal{P}=\{P_1,\cdots,P_d\}$ of $\pi^{-1}(w)$. Say $x|_{[1,N]}$ is in $P_1$. There is a preimage $x'$ of $y$ such that $x'|_{[1,N]}$ is in $P_1$. Take a bridge $\vec{u}$ from $x|_{[1,N]}$ to $x'|_{[1,N]}$. Then $z=x|_{(-\infty,0]}\vec{u}x'|_{[N+1,\infty)}$ is a preimage of $y$ left asymptotic to $x$ as desired. So $\pi$ is right continuing. In a similar way $\pi$ is left continuing and hence is bi-continuing.
  \end{proof}
  
  The conditions (\ref{itm:constant_depth}) and (\ref{itm:constant_t-depth}) are irrelevant of directions of transitions. So it is observed that $\pi$ is constant-class-to-one if and only if $d_\pi^l(y)=d_\pi$ for all $y$ in $Y$. The equivalence justifies to use the term simply {\em constant-class-to-one}, not emphasizing {\em right} or {\em left} like constant-{\em right}-class-to-one or constant-{\em left}-class-to-one.

  \begin{corollary}\label{cor:constant-class-to-one_iff_constant-left-class-to-one}
    A factor code from an irreducible shift of finite type onto a sofic shift is constant-class-to-one with $d_\pi=d$ if and only if it is constant-left-class-to-one.
  \end{corollary}
  
  Proposition~\ref{prop:right_class-closing_left_continuing_then_constant-class-to-one} proves the converse of Theorem~\ref{thm:constant-class-to-one}. Theorem~\ref{thm:bi-class-closing_bi-continuing_iff_constant-class-to-one} is a main result of this section. It is shown by Proposition~\ref{prop:right_class-closing_left_continuing_then_constant-class-to-one} combined with Theorem~\ref{thm:constant-class-to-one}.
  
  \begin{proposition}\label{prop:right_class-closing_left_continuing_then_constant-class-to-one}
    Let $\pi$ be a factor code from an irreducible shift of finite type $X$ onto a sofic shift $Y$. If $\pi$ is right class-closing and left-continuing then it is constant-class-to-one.
  \end{proposition}
  \begin{proof}
    Suppose that $\pi$ is right class-closing and left continuing but not constant-class-to-one. Choose a point $y$ in $Y$ with $d=d_\pi^r(y)>d_\pi$ and preimages $x_1,\cdots,x_d$ of $y$ such that $x_i\not\sim^rx_j$ for all $i\ne j$. Replacing $y|_{(-\infty,0]}$ with a left transitive tail get a left transitive point $y'$ right asymptotic to $y$. Since $y'$ is left transitive, we see that $d'=d_\pi^l(y')=d_\pi<d$. Since $\pi$ is left-continuing, for every $1\le j\le d$ there must be a preimage $x'_j$ of $y'$ right asymptotic to $x_j$. As $d'<d$ there must be $1\le i<k\le d$ such that $x'_i\sim^lx'_k$, i.e., $[x'_i]^l=[x'_k]^l$. However, as right asymptotic points of $x_i$ and $x_k$, respectively, they are not right equivalent and $[x'_i]^r\ne[x'_k]^r$. Since we assumed $\pi$ to be right class-closing, by (\ref{itm:left_class_in_right_class}) of Theorem~\ref{thm:class-closing_conditions} we meet a contradiction.
  \end{proof}
  
  \begin{theorem}\label{thm:bi-class-closing_bi-continuing_iff_constant-class-to-one}
    Let $\pi$ be a factor code from an irreducible shift of finite type $X$ onto a sofic shift $Y$. Then $\pi$ is right class-closing and left-continuing if and only if it is left class-closing and right-continuing if and only if it is constant-class-to-one.
  \end{theorem}
  
  Finite-to-one class-closing factor codes are exactly closing ones. Hence, a factor code is finite-to-one and constant-class-to-one factor code if and only if it is bi-closing and bi-continuing if and only if it is constant-to-one.
  
  Further, applying Theorem~\ref{thm:class-closing_is_continuing} we obtain the following Corollary~\ref{cor:bi-class-closing_sft_iff_constant-to-one} which generalises Theorem~7.3. in \cite{Nas83}.
  
  \begin{corollary}\label{cor:bi-class-closing_sft_iff_constant-to-one}
    Let $\pi$ be a factor code between irreducible shifts of finite type $X$ and $Y$. Then $\pi$ is bi-class-closing if and only if it is constant-class-to-one.
  \end{corollary}
  
  Until now some equivalences between closing codes, continuing codes and constant-to-one factor codes have been extended to infinite-to-one cases by Theorem~\ref{thm:bi-class-closing_bi-continuing_iff_constant-class-to-one} and Corollary~\ref{cor:bi-class-closing_sft_iff_constant-to-one}. Not as in finite-to-one cases, however, bi-continuing factor codes are not necessarily bi-class-closing nor holds the converse as already shown in Example~\ref{eg:continuing_but_not_class-closing} and in Example~\ref{eg:class-closing_but_not_continuing}, respectively.
  
  Of the two examples, Example~\ref{eg:class-closing_but_not_continuing} also provides a bi-class-closing factor code which is not constant-class-to-one. Its image is surely not of finite type but appears to be an {\em almost finite type} shift, which is defined to be the image of a bi-closing factor code from an irreducible shift of finite type. Many factor theorems between shifts of finite type are generalised to between almost finite type shifts \cite{Ash93,BoyMT87} and a sofic shift $Y$ is almost finite type if and only if among factor codes from irreducible shifts of finite type onto it there is a {\em minimal} one through which every such a factor code is factored \cite{BoyKM85}.
  
  We now show that a bi-class-closing factor of an irreducible shift of finite type is almost finite type. The graph $\bar{G}$ and the factor $\bar\pi$ from $\bar{G}$ given by the subset construction in the previous section are to be used again. As in Theorem~\ref{thm:class-closing_is_continuing} $\bar{G}$ actually denotes the sink $H$ of the subset construction graph of $G$ given in the section. The following proposition shows that $\bar{G}$ together with $\bar\pi$ is a bi-closing extension of $Y$.

  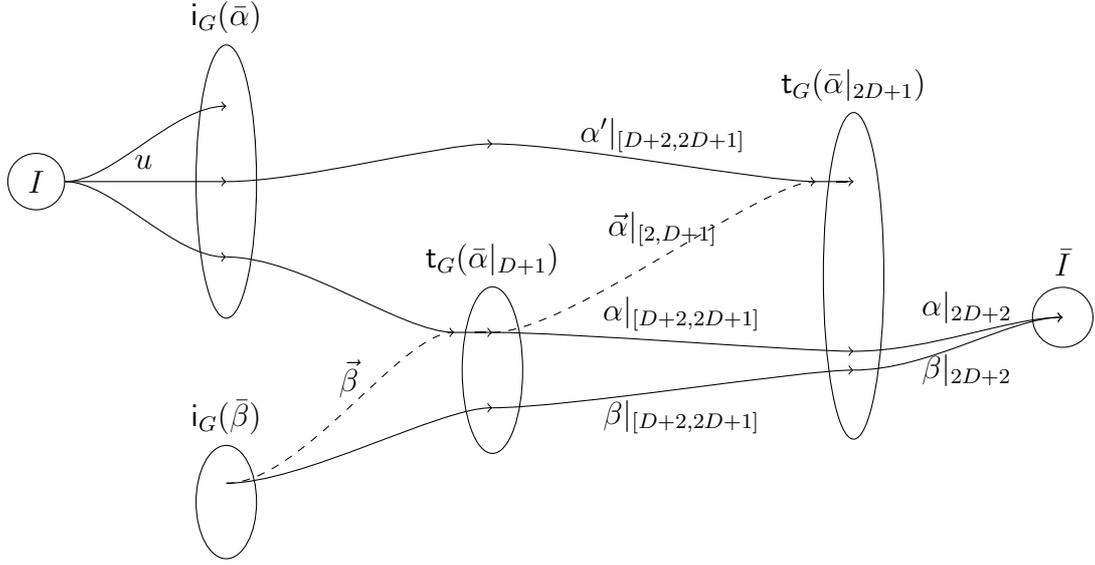
\begin{figure}
    \begin{tikzpicture}[->,to path={.. controls +(1,0) and +(-0.7,0) .. (\tikztotarget) \tikztonodes}]
      \node[circle,draw] (I) at (0,4) {$I$};
      \node[coordinate] (J1) at (2.5,3) {};
      \node[coordinate] (J2) at (2.5,4) {};
      \node[coordinate] (J3) at (2.5,5) {};
      \node[ellipse,fit=(J1) (J2) (J3),draw,inner sep=8pt,label=$\mathsf{i}_G(\bar\alpha)$] {};
      \node[coordinate] (J0) at (2.5,0) {};
      \node[coordinate] (J-) at (2.5,-.5) {};
      \node[ellipse,fit=(J-) (J0),draw,inner sep=8pt,label=$\mathsf{i}_G(\bar\beta)$] {};
      
      \draw (I) to (J1);
      \draw (I) to (J2);
      \draw (I) to node[below] {$u$} (J3);
      
      \node[coordinate] (K1) at (5.5,2) {};
      \node[coordinate] (K1') at (6,2) {};
      \node[coordinate] (K2) at (6,4.5) {};
      \draw (J1) to (K1) to (K1');
      \draw (J2) to (K2);
      \node[coordinate] (K0) at (6,1) {};
      \node[ellipse,fit=(K0) (K1'),draw,inner sep=8pt,label=$\mathsf{t}_G(\bar\alpha|_{D+1})$] {};
      \draw (J0) to (K0);
      \draw [dashed] (J0) to node[above] {$\vec\beta$} (K1);

      \node[coordinate] (L0) at (10.75,1.5) {};
      \node[coordinate] (L1) at (10.75,1.75) {};
      \node[coordinate] (L2) at (10.25,4) {};
      \node[coordinate] (L2') at (10.75,4) {};
      \draw (K0) to node[below] {$\beta|_{[D+2,2D+1]}$} (L0);
      \draw (K1') to node[above] {$\alpha|_{[D+2,2D+1]}$} (L1);
      \draw (K2) to node[above] {$\alpha'|_{[D+2,2D+1]}$} (L2) to (L2');
      \node[ellipse,fit=(L0) (L1) (L2'),draw,inner sep=8pt,label=$\mathsf{t}_G(\bar\alpha|_{2D+1})$] {};
      \draw [dashed] (K1') to node[above] {$\vec\alpha|_{[2,D+1]}$} (L2);

      \node[coordinate] (I*) at (13.5,2.2) {};
      \draw (L0) to node[below] {$\beta|_{2D+2}$} (I*);
      \draw (L1) to node[above] {$\alpha|_{2D+2}$} (I*);
      \node[ellipse,fit=(I*),draw,inner sep=8pt,label=$\bar{I}$] {};
    \end{tikzpicture}
    \caption{Left class-closing with delay $2D+1$}\label{fig:bi-class-closing_factors_are_almost_finite_type}
  \end{figure}
  
  \begin{proposition}\label{prop:bi-class-closing_image_has_a_bi-closing_cover}
    Let $G=(\mathcal{V},\mathcal{E})$ be an irreducible graph. Let $X=\mathsf{X}_G$ and $Y=\pi(X)$ where $\pi$ is a labelling map on $\mathcal{E}$. If $\pi$ is bi-class-closing with both right and left delays $D$, then $\bar\pi$ is left closing with delay $2D+1$.
  \end{proposition}
  \begin{proof}
    Let paths $\bar\alpha$ and $\bar\beta$ over $\bar{G}$ end at a vertex $\bar{I}$ and have the same label $w$ and the same length $2D+2$. To show that $\bar\alpha|_{2D+1}$ and $\bar\beta|_{2D+1}$ share the same terminal vertex on $\bar{G}$, we claim that
    \[ \mathsf{t}_{\bar{G}}(\bar\alpha|_{2D+1})=\mathsf{t}_G(\bar\alpha|_{2D+1})=\mathsf{t}_G(\bar\beta|_{2D+1})=\mathsf{t}_{\bar{G}}(\bar\beta|_{2D+1}) \]
    regarding $\bar\alpha$ and $\bar\beta$ as the sets of paths over $G$ at the same time.
    
    It follows from $\bar{I}=\mathsf{t}_G(\bar\alpha)=\mathsf{t}_G(\bar\beta)$ that for each $\alpha$ in $\bar\alpha$ there is $\beta$ in $\bar\beta$ with $\mathsf{t}_G(\beta)=\mathsf{t}_G(\alpha)$. Since $\pi$ is left class-closing with delay $D$ and $\alpha|_{[1,D+1]}$ and $\beta|_{[1,D+1]}$ are left accessible from $\mathsf{t}_G(\beta)=\mathsf{t}_G(\alpha)$ via $w|_{[D+2,2D+2]}$, the two blocks form a tangled set and there is a bridge $\vec\beta$ from $\beta|_{[1,D+1]}$ to $\alpha|_{[1,D+1]}$ (see Figure~\ref{fig:bi-class-closing_factors_are_almost_finite_type}).
    
    Say $\mathsf{i}_G(\bar\alpha)$ is right accessible from $I$ in $\mathcal{V}$ via $u$. Given $\alpha'$ in $\bar\alpha|_{[1,2D+1]}$, there is a bridge $\vec\alpha$ from $\alpha|_{[D+1,2D+1]}$ to $\alpha'|_{[D+1,2D+1]}$ as they are right accessible from $I$ via $uw|_{[1,D]}$ and $\pi$ is right class-closing with delay $D$. Then $\vec\beta\vec\alpha|_{[2,D+1]}$ is a bridge from $\beta|_{[1,2D+1]}$ to $\alpha'|_{[1,2D+1]}$. That is, there is a path over $G$ starting from a vertex in $\mathsf{i}_G(\bar\beta)$ and ending at $\mathsf{t}_G(\alpha'|_{2D+1})$ when $\alpha'$ is chosen arbitrarily in $\bar\alpha$. Hence $\mathsf{t}_G(\bar\alpha|_{2D+1})$ is a contained in $\mathsf{t}_G(\bar\beta|_{2D+1})$. A symmetric argument gives $\mathsf{t}_G(\bar\alpha|_{2D+1})=\mathsf{t}_G(\bar\beta|_{2D+1})$.
  \end{proof}
  
  \begin{corollary}\label{cor:bi-class-closing_image_is_aft}
    The image of a bi-class-closing factor code from an irreducible shift of finite type is almost finite type.
  \end{corollary}

  The conclusive relations between factor codes found in the section are summarised in Figure~\ref{fig:relations_in_general}.
  
  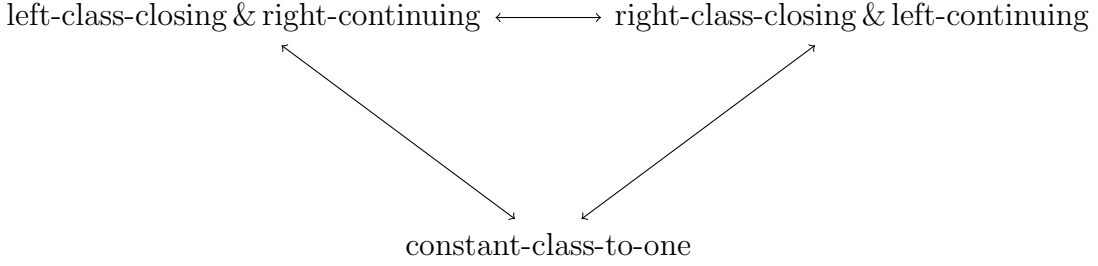
\begin{figure}
    \begin{tikzpicture}[->,auto,inner sep=5pt,node distance=30pt]
      \node (a) at (-4,3) {left-class-closing\,\&\,right-continuing};
      \node (b) at (4,3) {right-class-closing\,\&\,left-continuing};
      \node (c) at (0,0) {constant-class-to-one};      

      \path
        (a) edge[<->] (c)
        (b) edge[<->] (c)
        (a) edge[<->] (b);
    \end{tikzpicture}
    \caption{Relation between class-closing factor codes, continuing factor codes and constant-class-to-one factor codes}\label{fig:relations_in_general}
  \end{figure}
  
  We finish the paper with another observation on a property of class-closing factor codes  analogous to closing factor codes. It is about the points of $Y$ with degree more than $d_\pi$. Let $\pi$ be finite-to-one and $M_\pi=\{y\in Y\mid\pi^{-1}(y)>d_\pi\}$. Then $M_\pi$ is a subshift of $Y$ if and only if $\pi$ is bi-closing. The analogous conditions for class-closing factor codes are set in Proposition~\ref{prop:multiplicity_set}.
  
  \begin{proposition}\label{prop:multiplicity_set}
    Let $\pi$ be a factor code from an irreducible shift of finite type $X$ onto a sofic shift $Y$. Let $M_\pi^r=\{y\in Y\mid d_\pi^r(y)>d\}$ and $M_\pi^l=\{y\in Y\mid d_\pi^l(y)>d\}$ where $d=d_\pi$. Then $M_\pi^r$ (resp., $M_\pi^l$) is a subshift of $Y$ if and only if $\pi$ is right (resp., left) class-closing. Also $M_\pi^r=M_\pi^l$ if and only if $\pi$ is bi-class-closing. 
  \end{proposition}
  \begin{proof}
    Without loss of generality we will concern only right class-closing cases. Also, assume that $\pi$ is 1-block and $X$ is 1-step.
    
    ($\implies$) Assume that $\pi$ is not right class-closing. Then there are two left asymptotic points $x$ and $x'$ in $X$ with $x\not\sim^rx'$. Changing the common left infinite tail of the two points, we may assume that they are left transitive. Let $y=\pi(x)=\pi(x')$. It is left transitive, too. By Lemma~\ref{lmm:separated_classes} there are at least $d$ mutually separated right classes $C_1=[x]^r,C_2,\cdots,C_d$ over $y$. Since $[x']^r$ is not mutually separated from $C_1$, it is not any of $C_2,\cdots,C_d$. So $\llbracket y\rrbracket^r$ has at least $d+1$ right classes $C_1,\cdots,C_d,[x']^r$, and thus, $y$ is in $M_\pi^r$. As $y$ is left transitive, the $\sigma$-orbit closure of $M_\pi^r$ is the whole $Y$. For $M_\pi^r$ to be closed and $\sigma$-invariant it has to be $Y$ itself, which is impossible.
    
    ($\Longleftarrow$) We claim that for a right class-closing factor code $\pi$ a word with the depth $d$ occurs in $y$ if and only if $d_\pi^r(y)=d$. One direction is easy: by Theorem 4.22 of \cite{AllQ13}, given any point $y$ in $Y$, a word with the depth $\le d$ occurs in $y$ if $d_\pi^r(y)=d$. To show the converse, assume that a word $w$ with $d_\pi(w)=d$ occurs in $y$, say $y|_{[1,|w|]}=w$ and that $d_\pi^r(y)>d$. Then there must be two right classes over $y$ such that their points are routable through the same symbol at some $1<n<|w|$. Those two right classes end up not being mutually separated and it follows that $\pi$ is not right class-closing. Hence, the claim is shown to be true.
    
    Assume that $\pi$ is right class-closing. Then the set $M_\pi^r$ is given by an expression $\{y\in Y\mid d_\pi(y|_{[m,n]})>d\,\forall m,n\text{ with }m\le n\}$. That is, it is obtained by forbidding the words with the depth $d$ from $Y$ and is expressed as the intersection $\bigcap_{d_\pi(w)=d}\bigcap_{n\in\mathbb{Z}}\sigma^n(X\setminus[w]_0)$ where $[w]_0$ is a cylinder set at 0 determined by $w$. Here $\bigcap_{n\in\mathbb{Z}}\sigma^n(X\setminus[w]_0)$ expresses the set obtained by forbidding $w$ from $Y$. The latter intersection is closed and $\sigma$-invariant, and hence is a subshift of $Y$. So is $M_\pi^r$.
    
    Now, if $\pi$ is bi-class-closing, then both $M_\pi^r$ and $M_\pi^l$ are the same such subshift of $Y$. If $\pi$ is class-closing in only one direction of right or left, then only one of $M_\pi^r$ or $M_\pi^l$ is a subshift of $Y$ so that they cannot be the same. The last case reduces to when $\pi$ is neither left nor right class-closing. Note that for a factor code $\pi$ from an irreducible shift of finite type, no right transitive point belongs to $M_\pi^r$ and no left one to $M_\pi^l$ (see Theorem~\ref{thm:previous_results}). Moreover, if $\pi$ is not right class-closing, then $M_\pi^r$ contains some left transitive points (see the proof in ($\implies$)). Hence, if $\pi$ is neither left nor right class-closing, then $M_\pi^r\setminus M_\pi^l$ contains the same left transitive points that $M_\pi^r$ does. Therefore $M_\pi^r\setminus M_\pi^l$ is nonempty and $M_\pi^r\ne M_\pi^l$.
  \end{proof}
  
  \begin{acknowledgement*}
    The first author was supported by FONDECYT project 11140369, the second author by FONDECYT project 3130718 and the third author by the Basic Science Research Program through the National Research Foundation of Korea (NRF) funded by the Ministry of Education (2012R1A1A2006874).
  \end{acknowledgement*}

\bibliographystyle{amsplain}
\bibliography{ref}

\end{document}